\documentclass[12pt, reqno]{amsart}
\usepackage{amsmath, amsthm, amscd, amsfonts, amssymb, graphicx, color}
\usepackage[bookmarksnumbered, colorlinks, plainpages]{hyperref}
\usepackage{float}
\usepackage{graphicx,epstopdf}
\usepackage{bbm}
\usepackage{mathrsfs}
\usepackage{dsfont}
\usepackage{amssymb}
\usepackage{csquotes}
\usepackage[dvips]{epsfig}
\usepackage{latexsym}
\usepackage{exscale}
\usepackage[latin1]{inputenc}

\usepackage{tikz}
\usepackage{tikz-cd}
\usetikzlibrary{matrix,arrows,decorations.pathmorphing}

\hypersetup{colorlinks=true,linkcolor=red, anchorcolor=green, citecolor=cyan, urlcolor=red, filecolor=magenta, pdftoolbar=true}

\newtheorem{theorem}{Theorem}[section]
\newtheorem{lemma}[theorem]{Lemma}
\newtheorem{proposition}[theorem]{Proposition}
\newtheorem{corollary}[theorem]{Corollary}
\theoremstyle{definition}
\newtheorem{definition}[theorem]{Definition}
\newtheorem{example}[theorem]{Example}

\theoremstyle{remark}

\numberwithin{equation}{section}

\newcommand{\F}{\mathcal{F}}
\newcommand{\Lin}{\mathcal{L}}

\begin{document}

\setcounter{page}{1}

\setcounter{page}{1}

\title{Approximation spaces for H-operators}

\begin{center}
\author[A. G. AKSOY, D. A. THIONG ]{Asuman G\"{u}ven AKSOY, Daniel Akech Thiong}
\end{center}

\address{$^{*}$Department of Mathematics, Claremont McKenna College, 850 Columbia Avenue, Claremont, CA  91711, USA.}
\email{\textcolor[rgb]{0.00,0.00,0.84}{aaksoy@cmc.edu}}

\address{$^{1}$Department of Mathematics, Claremont Graduate University, 710 N. College Avenue, Claremont, CA  91711, USA..}
\email{\textcolor[rgb]{0.00,0.00,0.84}{daniel.akech@cgu.edu}}




\begin{abstract}
This paper defines and establishes relations among approximation spaces of certain operators called \textit{H-operators}, which generalize the notion of self-adjoint to Banach spaces.

\end{abstract} \maketitle

\section{Introduction} 
In the following, we give a brief review of the background, notation, and terminology that will be relevant to this paper.
Let $X$ and $Y$ be Banach (or quasi-Banach) spaces and $T:X \to Y$ be an operator. $\mathcal{L}(X, Y)$ denotes the normed vector space of all bounded operators from $X$ to $Y$ and $\mathcal{K}(X, Y)$ denote the collection of all compact operators from $X$ to $Y$. 

The problem of creating interpolation spaces is at the core of interpolation theory \cite{BL}.That is, given a pair $(X_{0}, X_{1})$ of Banach (or quasi-Banach) spaces, called a \textit{Banach couple}, with $X_{0}$ and $X_{1}$ both continuously embedded in some Hausdorff topological vector space, how can one construct and describe interpolation spaces $(X_{0}, X_{1})_{\theta, q} $ for the pair $(X_{0}, X_{1})$, where $\theta$ and $q$ are some parameters. Such spaces $(X_{0}, X_{1})_{\theta, q} $ should have the interpolation property that a linear operator $T$, which is bounded on $X_{i}$ for $i =0, 1$ is automatically bounded on $(X_{0}, X_{1})_{\theta, q} $. 

\

A natural question to ask is what properties of $T$ as a linear operator on $X_{i}$ still hold true when $T$ is viewed as a linear operator on $(X_{0}, X_{1})_{\theta, q} $. The answer to this classical question depends on the details of the method used to construct $(X_{0}, X_{1})_{\theta, q}$. Two of the main methods used are real and complex methods, but there are others. In \cite{CP}, it is shown using the real interpolation method that if $T \in \mathcal{K} (X_{0}, Y_{0}) $ and $T \in \mathcal{L} (X_{1}, Y_{1}) $, then $T \in \mathcal{K} ( (X_{0}, X_{1})_{\theta, q}, (Y_{0}, Y_{1})_{\theta, q})$. 

\

In this paper, we construct and describe interpolation spaces when $T$ is a compact H-operator. Under certain conditions regarding Bernstein and Jackson inequalities, interpolation spaces can be realized as approximation spaces, see \cite{Delo}, Theorem 9.1 on page 235. Thus, we are able to define approximation spaces for \textit{compact H-operators} using the sequences of their eigenvalues and establish relations among these spaces using interpolation theory. In section $1$, we start with a motivational example leading to the definition of H-operators. Section 2 discusses approximation spaces. Section 3 briefly presents interpolations spaces and illustrates how approximation spaces can be realized as examples of interpolation spaces under some conditions. Section 4 defines approximation spaces for compact H-operators, contrasts them with the general approximation spaces, and presents an inclusion theorem and a representation theorem. Section 5 points to a connection with Bernstein's Lethargy problem. 

\section{Compact H-Operators}

A fundamental result about linear operators on Hilbert spaces is the spectral theorem, which says that for a compact self-adjoint operator $T$ acting on a separable Hilbert space $H$, one can choose a system of orthonormal eigenvectors $\{v_{n} \}_{n\geq 1}$ of $T$ and corresponding eigenvalues $\{\lambda_{n}\}_{n \geq 1}$ such that \begin{equation} Tx = \sum_{n =1}^{\infty} \lambda_{n} \langle x, v_{n} \rangle v_{n}, \text{ for all } x \in H .\end{equation}  The sequence $\{\lambda_{n} \}$ is decreasing and, if it is infinite, converges to 0. 

\

To investigate the spectral properties of an arbitrary $T \in \mathcal{K}(H)$, where $H$ is a Hilbert space, it is useful to study the eigenvalues of the compact positive self-adjoint operator $T^{*}T$ associated with $T$. If $$\lambda_{1} (T^{*}T) \geq \lambda_{2} (T^{*}T) \geq \cdots > 0$$ denote the positive eigenvalues of $T^{*}T$, where each eigenvalue is repeated as many times as the value of its multiplicity, then the \textit{singular values} of $T$ are defined to be $$s_{n}(T) : = \sqrt{\lambda_{n} (T^{*}T)}, n \geq 1.$$

Using the representation (2.1) for $T^{*}T$ one can prove the following \textit{Schmidt representation} for $T$ \begin{equation} Tx = \sum_{n =1} ^{\infty} s_{n}(T) \langle \psi_{n}, x \rangle \phi_{n}, x \in H, \end{equation} 
where $\{\psi_{n} \}_{n \geq 1}$ and $\{\phi_{n} \}_{n \geq 1}$ are orthonormal systems in $H$ (see \cite{KV} ).

\

It is known that the concept of H-operators is the generalization in a Banach space of the concept of self-adjoint operators. We start with a motivation that will lead to a concrete definition. 

\begin{definition} 
A norm $\| \cdot \|$ on the $n \times n$ matrices is called a \textit{unitarily invariant norm} if $$\|UXV \| = \|X \|$$ for all $X$ and for all unitary matrices $U$ and $V$. 
\end{definition} 


\begin{example}
For $d_{i} \in \mathbb{R}, i = 1, 2, \cdots n$, let's consider $$T = \begin{bmatrix}
    d_{1} & & \\
    & \ddots & \\
    & & d_{n}
  \end{bmatrix}$$

We will examine the operator norm of the resolvent of $T$, $(T - \lambda I)^{-1}$, where $\lambda \in \mathbb{C}$ is in the spectral set of $T$. 
First, note that $$|d_{j} - \lambda |  = |d_{j} - (a +bi)| = \sqrt{ (d_{j} - a)^{2} + b^{2} } \geq \sqrt{b^{2}} = |b| \implies |d_{j} - \lambda |^{-1} \leq | b | ^{-1} = | \text { Im } \lambda |^{-1}.$$

Therefore, we have \begin{equation}  \|(T -\lambda I)^{-1}\| = \max \{|d_{1} - \lambda|^{-1}, \cdots, |d_{n} - \lambda|^{-1} \} \leq |\text{Im}(\lambda)|^{-1}  \end{equation}

If $T$ is self-adjoint, then one can choose an orthonormal basis such that there exists a unitary matrix $P$ and a diagonal matrix such that $T = PDP^{-1}$.

Since the operator norm or any norm defined in terms of singular values is unitarily invariant, we have the norm of $T$ equal to the norm of $D$. So that it suffices to compute the operator norm of a diagonal matrix with real diagonal entries as we have done above. 

Thus if $T$ is a self-adjoint operator, which must have real eigenvalues, between complex Banach spaces, then  \begin{equation}\|(T -\lambda I)^{-1}\| \leq |\text{Im}(\lambda)|^{-1} \end{equation}

which motivates the following definition.
 
 \end{example}
\begin{definition}
Let $T \in \mathcal{L}(X, Y)$ be a linear operator between arbitrary complex Banach spaces $X$ and $Y$. Then $T$ is an \textit{$H$-operator} if and only if its spectrum is real and its resolvent satisfies $$||(T - \lambda I)^{-1}|| \leq C|\text{ Im } \lambda |^{-1},$$ where $\text{ Im } \lambda \neq 0$. 
\end{definition} 
Here $C$ is independent of the points of the resolvent. An operator in Hilbert space is an H-operator with constant $C=1$ if and only if it is a self-adjoint operator.
  In \cite{Mar}, it is proved that closed operators with real eigenvalues are an example of $H$- operators. 
 
If $T$ is a compact $H$-operator, then $\{\lambda_{k}(T) \}$ denotes the sequence of eigenvalues of $T$, and each eigenvalue is repeated according to its multiplicity. We also assume that $\{\lambda_{k}(T) \}$ is ordered by magnitude, so that $| \lambda_{1}(T)| \geq | \lambda_{2}(T)| \geq \cdots$. 


In 1918, F. Riesz proved compact operators have at most countable set of eigenvalues, $\lambda_{n}(T)$, which arranged in a sequence, tend to zero. This result raises the following question:

\

\textit{What are the conditions on  $T\in \Lin (X, Y)$ such that $(\lambda_{n}(T)) \in \ell_{q}$? }

\

Having $(\lambda_{n}(T)) \in \ell_{1}$ is the precise condition one needs in order to generalize the \textit{Schmidt representation} (2.2) given above for compact operators on Banach spaces. 

\

 In this paper, we will be able to answer this question after carefully defining approximation spaces for compact H-operators. 
 
 The question can be recast more specifically, what is the rate of convergence to zero of the sequence $(\lambda_{n}(T))$?

\

Here is an example that shows the importance of the preceding question.

\begin{example}

Consider the diagonal operator $$T = \text{ diag } (a_{1}, a_{2}, a_{3} , \cdots,  )\quad \mbox{ where}\quad  a_{n} = \frac{1}{\log (n + 1)},\quad  n = 1, 2, \cdots. $$

Note that $T$ is compact and its eigenvalues are $\lambda_{n}(T) = a_{n}$. In \cite{KV}, it is shown that for each $q >0$, the number $a_{n}^{q} $ goes to zero slower than $\displaystyle \frac{1}{n}$ when $n \to \infty$. It follows that $\lambda_{n}(T) \notin \ell_{q}$. 
\end{example} 

\

To answer the question on the rate of convergence,  in \cite{Pieeig}, A. Pietsch developed  the theory of $s$-numbers $s_{n}(T)$ (closely related to singular values), which characterize the degree of compactness of $T$. 
There are several possibilities  of  assigning  to every operator $T: X \to Y$ a certain sequence of numbers $\{s_n(T)\}$ such that 
$$ s_1(T) \geq s_2(T) \geq \dots \geq 0.$$ The main examples of s-numbers to be used in this paper are approximation numbers and Kolmogorov numbers.  An operator $T\in \mathcal{L}(X,Y)$ has finite rank if  $\ rank(T) := \dim \{Tx:\,\, x\in X\}$ is finite.  For two arbitrary normed spaces $X$ and $Y$, we define the collection of the finite-rank operators as follows: $$\F(X, Y) = \{A \in \mathcal{L}(X, Y): \text{rank} (A) \leq n -1 \},$$ which forms the smallest ideal of operators.  As usual $\mathcal{K}(X, Y)$ is the collection of compact operators.


\begin{definition} We give the definition of the following s-numbers:
\begin{enumerate}
\item The \textit{nth approximation number} $$\alpha_{n}(T) = \inf\{||T - A||: A \in \F(X, Y)\},\quad n=0,1,\dots$$
Note that $\alpha_{n}(T)$ provides a measure of how well T can be approximated by finite mappings whose range is at most n-dimensional.  
The largest $s$-number is the approximation number. 


\item The \textit{nth  Kolmogorov diameter} of $T \in \mathcal{L}(X)$ is defined by $$\delta_{n}(T) = \inf \{||Q_{G} T||: \dim G \leq n \}$$ where the infimum is over all subspaces $G \subset X$ and $Q_{G}$ denotes the canonical quotient map $Q_{G}: X \rightarrow X/G$. 
\end{enumerate}
\end{definition} 
It is clear that $\alpha_n(T)$ and $\delta_{n}(T) $ are  monotone decreasing sequences and that 
$$ \lim_{n\to \infty} \alpha_n(T)=0 \quad\mbox{if and only if }\quad T\in \mathcal{F}(X,Y)$$ and 
$$ \lim_{n\to \infty} \delta_n(T)=0 \quad\mbox{if and only if }\quad T\in \mathcal{K}(X,Y).$$

In \cite{GKL}, it is shown that for any compact operator $T$ on a Hilbert space $H$ the n-th singular value $s_{n}(T)$ coincides with the n-th approximation number $\alpha_{n}(T)$. This allows us to compute $\alpha_{n}(T)$.

\begin{example}

Consider the non-self-adjoint $T = \begin{bmatrix} 2 & 1 & 0 \\ 0 & 2 & 0\\1 & 1 & 1 \end{bmatrix} $

The characteristic polynomial of $T$ is $\Delta (\lambda) = \lambda^{3}  - 5\lambda^{2} + 8 \lambda -4$ and the eigenvalues of $T$ are $\lambda_{1} (T) = 2$, $\lambda_{2} (T) = 2$ and $\lambda_{3} (T) = 1$. 

We have $T^{*} T =  \begin{bmatrix} 5 & 3 & 1 \\ 3 & 6 & 1\\1 & 1 & 1 \end{bmatrix} $.

The characteristic polynomial of $T^{*}T$ is $\Delta(\lambda) = \lambda^{3}  - 12\lambda^{2} + 30 \lambda -16$ and the eigenvalues are approximately $\lambda_{1} (T^{*}T) = 8.796$, $\lambda_{2} (T^{*}T) = 2.466$, and $\lambda_{3} (T^{*}T) = 0.738$

\end{example}

Note that we have $\alpha_{1}(T) = s_{1}(T) > |\lambda_{1}(T)|$, $\alpha_{2}(T) = s_{2}(T) < |\lambda_{2}(T)|$, and $\alpha_{3}(T) = s_{3}(T) < |\lambda_{3}(T)|$, which means we cannot compare $\alpha_{n}(T)$ with $|\lambda_{n}(T)|$.

Note also that in the preceding example, $T \neq T^{*}$ and $T^{*}T \neq TT^{*}$, that is, $T$ is neither self-adjoint nor normal. We should give an example that is not self-adjoint, but normal and see how it compares.

\begin{example}
Consider $T=\begin{bmatrix}2&\!-3\\3&\;2\end{bmatrix} $.
Then $T^*=T^t=\begin{bmatrix}\;2&3\\\!-3&2\end{bmatrix}$. 

We have $T^{*}T  =\begin{bmatrix}13&\!0\\0&\;13\end{bmatrix}  = TT^{*}$.

The eigenvalues of $T$ are $\lambda_{1} = 2 + 3i$ and $\lambda_{2} = 2 - 3i$ where it follows that $|\lambda_{1}| = \sqrt{13} = |\lambda_{2} |$ . The singular values of $T$ are $s_{1}(T) = \sqrt{13} = s_{2} (T)$. 

In this case, we have $\alpha_{n}(T) = s_{n}(T) = |\lambda_{n}(T)|$, for $n \in \{1, 2\}$. 
\end{example}

For compact self-adjoint operators on Hilbert spaces and more broadly for compact $H$- operators we will be able to always compare the preceding approximation quantities. 

Indeed, the importance of $H$-operators comes from a result of Markus \cite{Mar}, which shows that for a compact $H$-operator $T$ with eigenvalues $(\lambda_{n})$ (numbered in order of decreasing modulus and taking into account their multiplicity), the sequence $(| \lambda_{n} (T)|)$ is equivalent to \textit{approximation numbers} and \textit{Kolmogorov diameters}. Specifically, in \cite{Mar}  Markus proved the following theorem. 
\begin{theorem}
If $T$ is a compact H-operator on a Banach space $X$, then
\begin{equation} \delta_{n-1} (T) \leq \alpha_{n}(T) \leq 2 \sqrt{2} C |\lambda_{n}(T)| \leq 8C(C + 1) \delta_{n-1}(T), \end{equation} where $C$ is a constant from the definition of $H$-operator,  
and $\delta_n(T)$ and $\alpha_n(T)$ are the n-th Kolmogorov diameter and n-th approximation numbers of $T$ respectively. 

\end{theorem}
The following corollary follows from the preceding theorem.
\begin{corollary} 
If $T$ is a compact H-operator on a Banach space $X$, then for any $0 < \mu \leq \infty$, $|\lambda_{n}(T)|  \in \ell_{\mu}  \iff \delta_n(T)  \in \ell_{\mu} \iff \alpha_n(T) \in \ell_{\mu}$. 

\end{corollary}

This equivalence allows us to construct approximation spaces for $H$- operators using sequences of eigenvalues, but before we do this we will first introduce approximation spaces. 

\section{Approximation spaces} 

\begin{definition} 
A \textit{quasi-norm} is a non-negative function $||.||_{X}$ defined on a real or complex linear space $X$ for which the following conditions are satisfied:
\begin{enumerate}
\item [(1)] If $||f||_{X} = 0$ for some $f \in X$, then $f =0$.
\item[(2)] $|| \lambda f||_{X} = | \lambda | ||f||_{X}$ for $f \in X$ and all scalars $\lambda$.
\item [(3)] There exists a constant $c_{X} \geq 1$ such that $$||f + g||_{X} \leq c_{X} [||f||_{X} + ||g||_{X}]$$ for $f, g \in X$. 
\end{enumerate} 
\end{definition} 

A \textit{quasi-Banach space} is any linear space $X$ equipped with a quasi-norm $||.||_{X}$ such that every Cauchy sequence is convergent.

\begin{definition} \label{ap_sch} An \textit{approximation scheme} $(X, A_{n})$ is a quasi-Banach space $X$ together with a sequence of subsets $A_{n}$ satisfying the following:
\begin{itemize}
\item[$(A1)$] there exists a map $K:\mathbb{N}\to\mathbb{N}$ such that $K(n)\geq n$ and $A_n+A_n\subseteq A_{K(n)}$ for all $n\in\mathbb{N}$,

\item[$(A2)$] $\lambda A_n\subset A_n$ for all $n\in\mathbb{N}$ and all scalars $\lambda$,

\item[$(A3)$] $\bigcup_{n\in\mathbb{N}}A_n$ is a dense subset of $X$.
\end{itemize}

\end{definition}  
Approximation schemes were introduced in Banach space theory by Butzer and Scherer in $1968$ \cite{But} and independently by Y. Brudnyi and N. Kruglyak under the name of ``approximation families''  in \cite{BK}. They were
popularized by Pietsch in his 1981 paper \cite{Pi}, for  later developments we refer the reader to  \cite{  Ak-Al, AA,  AL}.

Let $(X, A_{n})$ be an approximation scheme. For $f \in X$ and $n =1, 2, \cdots$, the \textit{nth approximation number} is defined by $$\alpha_{n} (f, X) : = \inf \{||f - a|||_{X}: a \in A_{n -1}\}.$$  $\alpha_{n}(f, X)$ is the error of best approximation to $f$ by the elements of $A_{n-1}$. 

\begin{definition} \label{ap_sp}
Let $ 0 < \rho < \infty$ and $ 0 < \mu \leq \infty$. Then the \textit{approximation space}  $X_{\mu}^{\rho}$, or more precisely $(X, A_{n})_{\mu}^{\rho}$ consists of all elements $f \in X$ such that $$(n^{\rho - \mu^{-1} } \alpha_{n}(f, X) )\in \ell_{\mu},$$
where $n = 1, 2, \cdots$. We put $||f||_{X_{\mu}^{\rho}} = ||n^{\rho - \mu^{-1} } \alpha_{n}(f, X) )||_{\ell_{\mu}}$ for $f \in X_{\mu}^{\rho}$. 
\end{definition} 

Now, we define and present Lorentz sequences as examples of approximation spaces. 

\begin{definition}
A null sequence $x = (\zeta_{k}) $ is said to belong to the \textit{Lorentz sequence space} $\ell_{p,q}$ if the non-increasing re-arrangement $(s_{k}(x))$ of its absolute values $|\zeta_{k}|$ satisfies

\begin{equation}
\left(k^{\frac{1}{p} - \frac{1}{q}} s_{k}(x) \right) \in \ell_{q}, 
\end{equation} 
so that 

\begin{equation}
\lambda_{p, q}(x) = \begin{cases} \left( \displaystyle\sum_{k=1}^{\infty} \left( k^{\frac{1}{p} - \frac{1}{q}} s_{k}(x) \right)^{q} \right)^{\frac{1}{q}} \quad \text{ for } 0 < p < \infty \text { and } \quad 0 < q < \infty \\
\sup_{1 \leq k < \infty} k^{\frac{1}{p}} s_{k}(x) \quad \quad\quad \text{ for } 0 < p < \infty \text{ and } \quad q = \infty
\end{cases} 
\end{equation} 
is finite. 
\end{definition}

\begin{example}

Let $ 0 < \rho < \infty$ and $ 0 < \mu \leq \infty$. Consider the approximation scheme  $(X, A_{n})$, where $X = \ell_{\infty}$ and $A_{n}$:= the subset of sequences having at most $n$ coordinates different from $0$. 
For any $\eta \in \ell_{\infty}$, the sequence $\alpha_{n}(\eta; \ell_{\infty} )$ is the non-increasing rearrangement of the sequences $\eta$ and $X_{\mu}^{\rho} = \ell_{\rho^{-1}, \mu}$ (see\cite{Pi}, page 123). 

\end{example}

 \section{Interpolation spaces} 

\begin{definition}

An \textit{intermediate space} between $X_{0}$ and $X_{1}$ is any normed space $X$ such that $X_{0} \cap X_{1}  \subset X \subset X_{0} + X_{1}$ (with continuous embedding).
\end{definition} 

\begin{definition} 
An \textit{interpolation space} between $X_{0} $ and $X_{1}$ is any intermediate space $X$ such that every linear mapping from $X_{0} + X_{1}$ into itself which is continuous from $X_{0}$ into itself and from $X_{1}$ into itself is automatically continuous from $X$ into itself. An interpolation space is said to be of exponent $\theta$ $(0 < \theta < 1)$ , if there exists a constant $C$ such that one has $$||A||_{L(X)} \leq C ||A||_{L(X_{0})} ^{1- \theta} ||A||_{L(X_{1})}^{\theta} \text{ for all  } A \in L(X_{0}) \cap L(X_{1}).$$

\end{definition}

\begin{definition} 

Let $X_{i}, i =0, 1$ be two normed spaces, continuously embedded into a topological vector space $V$ so that $X_{0} \cap X_{1} $ and $X_{0} + X_{1}$ are defined with  $X_{0} \cap X_{1}$ equipped with the norm $$||f|| _{X_{0} \cap X_{1}} = \max \{ ||f||_{X_{0}}, ||f||_{X_{1}} \} $$ and $X_{0} + X_{1} $ is equipped with the norm $$||f||_{X_{0} + X_{1}} = \inf_{f = f_{0} + f_{1}} (||f_{0}||_{X_{0}} + ||f_{1}||_{X_{1}} ). $$
\end{definition} 

\begin{definition} 
For $f \in X_{0} + X_{1}$ and $t > 0$ one defines \[K(f, t) = \inf_{f = f_{0} + f_{1}} (||f_{0}||_{X_{0}} + t||f_{1}||_{X_{1}} ),\] and for $0 < \theta < 1$ and $1 \leq p \leq \infty$ (or for $\theta = 0, 1$ with $p = \infty$), one defines the \textit{real interpolation space} as follows: \[(X_{0}, X_{1})_{\theta, p} := \left\{f \in X_{0} + X_{1} : \quad t^{-\theta} K(f, t) \in L_{p} \left([0, \infty), \frac{dt}{t}\right) \right \} \] with the norm 
$||f||_{(X_{0}, X_{1} )_{\theta, p}}: =  \begin{cases} \left( \int_{0}^{\infty} \left(t^{-\theta} K(f, t) \right)^{p} \frac{dt}{t} \right)^{\frac{1}{p}}, \quad  0 < p < \infty \\ \sup_{0 \leq t < \infty} t^{-\theta} K(f, t), \quad\quad   p = \infty \end{cases}$
\end{definition}

 $K(f,t)$ is continuous and monotone decreasing in $t$, with $K(f,t)  \to 0$ as $ t\to 0+$. K-functional provides a relationship between interpolation and approximation spaces. 
 
 Once again, the Lorentz sequences are examples of interpolation spaces as can be seen from this classical example. 
 
 \begin{example} 
The Lorentz sequence space $l_{p, q}$ can be created as an approximation space using a real interpolation space. Take $X = \ell_{r}$ and $Y = \ell_{s}$. Then $(\ell_{r}, \ell_{s}) _{\theta, q} = \ell_{p, q}$ for $\frac{1}{p}: = \frac{(1 - \theta)}{r} + \frac{\theta}{s} $ and $ 0 < q \leq \infty$. 
\end{example} 

Now, if we take $p = q$, then $(\ell_{r}, \ell_{s}) _{\theta, p} = \ell_{p, p} = \ell_{p}$ for $\frac{1}{p}: = \frac{(1 - \theta)}{r} + \frac{\theta}{s} $ and $ 0 < p \leq \infty$, so that every $\ell_{p}$ may be realized as an interpolation space.


\

\

The real interpolation method provides the connection between interpolation spaces and approximation spaces. To state this connection, we need the following fundamental inequalities \cite{BS}.

Jackson's inequality, which measures the rate of decrease of $\alpha_{n}(f; X)$ is given by \begin{equation} \alpha_{n} (f, X) \leq C(n + 1)^{-\sigma} \alpha_{n}(f; Y), \text{ where } C  \text { is a constant. } \end{equation}

Berstein's inequality, which measures the rate of increase of $\|p_{n}\|_{X}$, where $p_{n} \in A_{n}$, is given by \begin{equation} \| p_{n} \|_{X} \geq C(n + 1)^{-\sigma} \| p_{n} \|_{Y}, \text{ where } C  \text { is a constant. }\end{equation}

If Jackson's and Bernstein's inequalities are valid for the pair $X$ and $Y$, then we can characterize completely the approximation spaces $X_{\mu}^{\rho}$ using the real interpolation spaces $(X, Y)_{\theta, q}$. 

\begin{proposition} [\cite{Delo}, Theorem 9.1] 

If both the Jackson and the Bernstein inequalities hold for the spaces $X$ and $Y$, then for $0 < \rho < r$ and $0 < \mu \leq \infty$ we have 
\begin{equation} X_{\mu}^{\rho} = (X, Y)_{\frac{\rho}{r}, \mu} . \end{equation} 

\end{proposition} 

According to the preceding proposition, to identify for a given $X$ and approximation scheme $Q$, the approximation spaces $X_{\mu}^{\rho}$, $0 < \mu < r$, it is enough to find a space $Y$ for which the Jackson and Bernstein inequalities are valid. 

There is a way to find such spaces $Y$:

\begin{proposition} [\cite{Delo}, Theorem 9.3] 
Consider an approximation scheme $(X, A_{n})$. Then for $0 < \mu \leq \infty$, $0 < \rho < \infty $, the space $Y : = Y_{\rho} : = X_{\mu}^{\rho}$ satisfies the Jackson and the Bernstein inequalities. Moreover, for $0 < \alpha < r$ and $0 < \mu_{1} \leq \infty$, 
\begin{equation}
(X, Y)_{\frac{\alpha}{r}, \mu_{1}} = X_{\mu_{1}}^{\alpha}  .
\end{equation} 
\end{proposition} 

We will use the preceding theorem to establish relations among approximation spaces of \textit{H-operators}.

\section{Main Results} 

Now, we define an approximation space for compact $H$-operator by using $|\lambda_{n}(T)| $. 

\begin{definition}
Let $ 0 < \rho < \infty$ and $ 0 < \mu \leq \infty$. Set $X: =\text {the set of all compact H-operators }$ between two arbitrary Banach spaces.Consider an approximation scheme $(X, A_{n})$.  We define an approximation space for $H$-compact operators by  $$A_{\mu}^{\rho}: = \{T \in X: (n^{\rho - \mu^{-1} } |\lambda_{n}(T)| ) \in \ell_{\mu} \}\quad n=1,2,\dots$$   We put $||T||_{A_{\mu}^{\rho}} = ||n^{\rho - \mu^{-1} } |\lambda_{n}(T)| )||_{\ell_{\mu}}$ for $T\in A_{\mu}^{\rho}$. 

\end{definition} 

To realize the importance of constructing approximation spaces for compact H-operators, let's recall the following question: What are the conditions on  $T\in \Lin (X, Y)$ such that $(\lambda_{n}(T)) \in \ell_{q}$?  

Answer: If $T$ is a compact H-operator, then $A_{q}^{\frac{1}{q}}$ consists of all elements $T\in \Lin (X, Y)$ such that $(|\lambda_{n}(T))| \in \ell_{q}$, which implies $(\lambda_{n}(T)) \in \ell_{q}$.

\begin{lemma}
If $0 < \theta < 1$ and $1 \leq \mu_{1}  \leq \mu_{2}  \leq \infty$, one has $(X_{0}, X_{1})_{\theta, \mu_{1}} \subset (X_{0}, X_{1})_{\theta, \mu_{2}} $ (with continuous embedding). 
\end{lemma} 
\begin{proof}

Note that if $1 \leq \mu < \infty$, and $t_{0} > 0$, then by the monotonicity of the K-functional, $K(f, t)$, we have $K(f, t) \geq K(f, t_{0}) \text{ for } t > t_{0}$, so that 

$$\left(||f||_{(X_{0}, X_{1})_{\theta, \mu} }\right)^{\mu}  =  \int_{0}^{\infty} \left(t^{-\theta} K(f, t) \right)^{\mu} \frac{dt}{t}\  \geq \left [K(f, t_{0}) \right]^{\mu} \int_{t_{0}}^{\infty} t^{-\theta \mu} \frac{dt}{t}  = \left [K(f, t_{0}) \right]^{\mu} \frac{t_{0}^{-\theta \mu}} {\theta \mu} $$
  which implies 
$$ t_{0}^{-\theta} K(f, t_{0}) \leq C ||f||_{(X_{0}, X_{1})_{\theta, \mu} } $$ and thus we have $ ||t^{-\theta} K(f, t)||_{L_{\infty}((0, \infty), \frac{dt}{t})} \leq C ||f||_{(X_{0}, X_{1})_{\theta, \mu} }$
and now using the  H\"older's inequality, one obtains: $$||f||_{(X_{0}, X_{1})_{\theta, \mu_{2}} } = ||t^{-\theta} K(f, t)||_{L_{\mu_{2}}((0, \infty), \frac{dt}{t})} \leq C' ||f||_{(X_{0}, X_{1})_{\theta, \mu_{1}} } \text{ for } 1 \leq \mu_{1}  \leq \mu_{2}  \leq \infty.$$
\end{proof}

We will also use the following lemma, which gives the basic relation between Kolmogorov numbers and approximation numbers. 
\begin{lemma} [\cite{CS}] 
$\delta_{n}(T) \leq \alpha_{n}(T)$ for all $T \in \mathcal{L}(X,Y)$. 
\end{lemma} 



\begin{theorem} [Inclusion Theorem]

Let $ 0 < \rho < \infty$ and $ 0 < \mu_{1} \leq \mu_{2} \leq \infty$. If $T$ is a compact H-operator between arbitrary Banach spaces $X$ and $Y$, then $A_{\mu_{1}}^{\rho} \subset A_{\mu_{2}}^{\rho }$. 

\end{theorem}

\begin{proof}

From (2.5) we have $\alpha_{n}(T) \leq 2 \sqrt{2} C \left |\lambda_{n}(T) \right| $, which implies $$n^{\rho - \mu^{-1} } \alpha_{n}(T) \leq 2\sqrt{2} C n^{\rho - \mu^{-1}}  \left |\lambda_{n} (T)  \right | . $$

Thus, if $n^{\rho - \mu^{-1}}  \left |\lambda_{n} (T)  \right |  \in \ell_{\mu}$, then $2\sqrt{2} C n^{\rho - \mu^{-1}}  \left |\lambda_{n} (T)  \right | \in \ell_{\mu}$, which implies that $n^{\rho - \mu^{-1} } \alpha_{n}(T)  \in \ell_{\mu}$. 
It follows that if $T \in A_{\mu}^{\rho}$, then $T \in X_{\mu}^{\rho}$. Therefore, $A_{\mu}^{\rho} \subset X_{\mu}^{\rho}$.

By Lemma 5.3, we have $\delta_{n-1} (T) \leq \alpha_{n +1} (T)$, (2.5) implies that $$2 \sqrt{2} C |\lambda_{n}(T)| \leq 8C(C + 1) \delta_{n-1}(T) \leq 8C(C + 1) \alpha_{n}(T).$$ Hence, $$\frac{2 \sqrt{2} C}{ 8C(C + 1)}  n^{\rho - \mu^{-1}}  |\lambda_{n}(T)| \leq n^{\rho - \mu^{-1}}  \alpha_{n}(T).$$Thus, if $n^{\rho - \mu^{-1}}  \alpha_{n}(T) \in \ell_{\mu}$, then $\frac{2 \sqrt{2} C}{ 8C(C + 1)}  n^{\rho - \mu^{-1}}  |\lambda_{n}(T)| \in \ell_{\mu}$. It follows that $T \in X_{\mu}^{\rho} \implies T \in A_{\mu}^{\rho}$. 
Hence, $X_{\mu}^{\rho} \subset A_{\mu}^{\rho}$.  

We have $A_{\mu}^{\rho} = X_{\mu}^{\rho}$. By Proposition 4.6, we also know  $A_{\mu}^{\rho} = X_{\mu}^{\rho} = (X, Y)_{\frac{\rho}{r}, \mu}$. 

By Lemma 5.2, we have $A_{\mu_{1} }^{\rho} \subset A_{\mu_{2}}^{\rho }$ as  was promised.

\end{proof}

The proof of the following theorem relies on Markus' inequality (2.5), H\"{o}lder's inequality and proof of an analogous representation theorem in \cite{Pi}. 

\begin{theorem} [Representation Theorem]

Let $ 0 < \rho < \infty$ and $ 0 < \mu \leq \infty$. Set $X: =\text {the set of all compact H-operators }$ between two arbitrary Banach spaces and $A_{\mu}^{\rho}: = \{T \in X: (n^{\rho - \mu^{-1} } |\lambda_{n}(T)| ) \in \ell_{\mu} \}$. 
Consider an approximation scheme $(X, A_{n})$. Then $T \in X$ belongs to $A_{\mu}^{\rho}$ if and only if there exists $g_{n} \in A_{2^{n}}$ such that $T = \sum_{n =0}^{\infty} g_{n}$ and $(2^{n\rho} ||g_{n}||_{X} ) \in \ell_{\mu}$. 
Moreover, $||T||_{A_{\mu}^{\rho}} ^{\text{rep}} : = \inf||(2^{n\rho}||g_{n}||)||_{\ell_{\mu}}$, where the infimum is taken over all possible representations, defines an equivalent quasi-norm on $A_{\mu}^{\rho}$.

\end{theorem}

\begin{proof}

Suppose $T \in A_{\mu}^{\rho}$. We wish to find $g_{n} \in A_{2^{n}}$ such that $T = \sum_{n =0}^{\infty} g_{n}$ and $(2^{n\rho} ||g_{n}||_{X} ) \in \ell_{\mu}$. Choose $g_{n}^{\star}  \in A_{2^{n} -1} $ such that $$||T - g_{n}^{\star} ||_{X} \leq 2 \alpha_{2^{n}}(T) \leq 4\sqrt{2} C|\lambda_{2^{n}}(T)|.$$
Set $g_{0} = 0 = g_{1}$, and $g_{n +2} = g_{n+1}^{\star}  - g_{n}^{\star} $ for $n = 0, 1, \cdots.$. We have $g_{n} \in A_{2^{n}}$, and $$T = \lim_{n \to \infty} g_{n}^{\star} =  \sum_{n =0}^{\infty} g_{n}.$$

Moreover, it follows from $$||g_{n+2}||_{X} \leq c_{X} [||T - g_{n+1}^{\star} ||_{X} + ||T-g_{n}^{\star}||_{X} ] \leq 4c_{X} \alpha_{2^{n}}(T) \leq 16\sqrt{2} c_{X}C|\lambda_{2^{n}}(T)|$$ that $(2^{n\rho} ||g_{n}||_{X} ) \in \ell_{\mu}$. 

Next, suppose there exists $g_{n} \in A_{2^{n}}$ such that $T = \sum_{n =0}^{\infty} g_{n}$ and $(2^{n\rho} ||g_{n}||_{X} ) \in \ell_{\mu}$. We must show that $T \in A_{\mu}^{\rho}$. Although $\alpha_{n}(T, X)$ is in general not a continuous function of $T$, we can also find an equivalent quasi-norm $X$, $p$-norm, that is always continuous. Thus, we can assume that $||.||_{X}$ is a p-norm with $0 < p < \mu$. If $T \in X$ can be written in the form $T = \sum_{n =0}^{\infty} g_{n}$ such that $g_{n} \in A_{2^{n}}$ and $(2^{n\rho} ||g_{n}||_{X} ) \in \ell_{\mu}$, then it follows from $\sum_{n =0}^{N-1} g_{n} \in A_{2^{N}-1} $ that 
\begin{flalign*}
|\lambda_{2^{N}}(T)| & \leq 2\sqrt{2} (C + 1) \alpha_{2^{N}}(T) \leq 2\sqrt{2} (C + 1)|| T - \sum_{n =0}^{N-1} g_{n}||_{X}^{p}&\\  
& \leq 2\sqrt{2} (C + 1)\sum_{n =N}^{\infty} ||g_{n}||_{X}^{p}.&
  \end{flalign*}

In the case $0 < \mu < \infty$ we put $q = \frac{\mu}{p}$, and choose $\gamma$ such that $\rho p > \gamma > 0$. 
Then 

\begin{flalign*}
\sum_{N =0}^{\infty} [2^{N\rho} \lambda_{2^{N}}(T) ]^{\mu}& \leq 2\sqrt{2} (C + 1)\sum_{N =0}^{\infty} [2^{N\rho} \alpha_{2^{N}}(T) ]^{\mu} &\\ \leq  & 2\sqrt{2} (C + 1)\sum_{N =0}^{\infty}   2^{N\rho \mu} \left( \sum_{n =N}^{\infty} 2^{-n\gamma}2^{n\gamma} ||g_{n}||_{X} ^{p} \right) ^{q} & \\
& \leq 2\sqrt{2} (C + 1)\sum_{N =0}^{\infty}   2^{N\rho \mu} \left( \sum_{n =N}^{\infty} 2^{-n\gamma q'} \right)^{\frac{q}{q'} } \left( \sum_{n = N}^{\infty} 2^{n\gamma q} ||g_{n}||_{X} ^{p} \right) &\\
& \leq c_{1} 2\sqrt{2} (C + 1) \sum_{N =0} ^{\infty} 2^{N(\rho \mu - \gamma q)} \sum_{n =N} ^{\infty} 2^{n \gamma q }||g_{n} ||_{X}^{\mu} &\\
& \leq c_{1} 2\sqrt{2} (C + 1) \sum_{n=0}^{\infty} 2^{n\gamma q} ||g_{n}||_{X}^{\mu} \sum_{N=1}^{n} 2^{N(\rho \mu - \gamma q)} &\\
& \leq c_{2} 2\sqrt{2} (C + 1)  \sum_{n =0}^{\infty} [2^{n \rho} ||g_{n}||_{X} ]^{\mu} < \infty&
 \end{flalign*} 
The desired result follows.

\end{proof}

\section{Connection to Bernstein's Lethargy Theorem} 

The question of the rate of convergence of $\lambda_{n}(T)$ provides some connection to the classical Bernstein's Lethargy problem. 
 
 Now, we consider the \textit{Bernstein lethargy problem} for linear approximation: given a nested system $A_{1} \subset A_{2} \subset \cdots $ of linear subspaces of a Banach space $X$ and a strictly decreasing sequence $d_{0} > d_{1} > \cdots > d_{n} \to 0$, does there exist an element $x \in X$  such that for all $n = 0, 1, 2, \cdots$,  $\alpha_{n} (x) = \alpha_{n}(x, A_{n}) = d_{n}$?
 
 The answer is \textit{yes} in many particular cases: if $X$ is a Hilbert space; if all $A_{n}$ are finite-dimensional; if $d_{n} > \sum_{k = n + 1}^{\infty} d_{k} $ for all $n$. However, the Bernstein problem is still unsolved in its general setting.
 
 In the remaining part of this section, we investigate for infinite-dimensional Banach spaces $X$ and $Y$ the existence of an operator $T \in \mathcal{L}(X, Y)$ whose sequence of approximation numbers $\{ \alpha_{n} (T) \}$ behaves like the prescribed sequence $\{d_{n} \}$ given above in the Bernstein lethargy problem. If $\mathcal{A}_{n}$ denotes the space of all bounded linear operators from $X$ into $Y$  with rank at most $n$, then $\alpha_{n}(T) = \rho (T, \mathcal{A}_{n})$.

 \begin{definition}
 
 The operator $T \in \mathcal{K}(X, Y)$ where $X$ and $Y$ are complex Banach spaces is said to be a \textit{kernel operator} if it can be represented in the form  \begin{equation} T = \sum_{j=1}^{\infty} \alpha_{j} f_{j}(\cdot) y_{j} \end{equation} ($f_{j} \in X^{*}, y_{j} \in X, ||f_{j}|| = ||y_{j} || = 1, j = 1, 2, \cdots $),
 where $\sum_{j =1}^{\infty} |\alpha_{j} |< \infty$. 
 \end{definition}
 
 \begin{proposition} \cite{Mar} 
 For any sequence of non-negative numbers $(d_{n})$ that tends to zero, a kernel operator $T$ exists such that $\delta_{n}(T) \geq d_{n} $ ($n =0, 1, 2, \cdots$). 
 \end{proposition} 
 
 By Lemma 5.3, we always have $\alpha_{n}(T) \geq \delta_{n}(T)$ for every $T \in \mathcal{L}(X, Y)$. Then by the preceding proposition, we have for a strictly decreasing sequence $d_{0} > d_{1} > \cdots > d_{n} \to 0$, there is always an element in $\mathcal{L}(X, Y)$, namely a kernel operator $T$ such that one has $\alpha_{n}(T) \geq \delta_{n}(T) \geq d_{n}$.


 \bibliographystyle{amsplain}

\begin{thebibliography}{99}
\normalsize
\bibitem{Ak-Al} A. G. Aksoy,  \textit{Q-Compact sets and Q-compact maps}, Math. Japon. \textbf{36} (1991), no. 1, 1-7.


\bibitem{AA} A. G.  Aksoy, J. M.  Almira, \textit{On approximation schemes and compactness} Proceedings of the first conference on classical and functional analysis, 5-24, Azuga-Romania, (2014).
\bibitem {AL} J. M. Almira and U. Luther,\textit{Compactness and generalized approximation spaces}, Numer. Funct. Anal. and Optimiz., 23, (202) 1-38.
\bibitem{BL} J. Bergh and J. Lofstrom, \textit{Interpolation spaces}, An Introduction, Grundlehren de Mathematischen Wissenschaften, No. 223. Spring-Verlag, Berlin - New York, 1976
\bibitem{BK} Y. Brudnij and N. Krugljak, \textit{ On a family of approximation spaces}, Investigation in function theory of several real variables, Yaroslavl State Univ., Yaroslavl, (1978), 15-42.
\bibitem{BS} P. L. Butzer and K. Scherer,\textit{ On the fundamental approximation theorems of D. Jackson and S. N. Bernstein and the theorems of M. Zamansky and S. B. Steckin}, Aequationses Math. \textbf{3 }(1969), 170-185 
\bibitem{But} P. L. Butzer and K. Scherer,\textit{ Approximations Prozesse und Interpolations methoden}, Biliographisches Inst. Mannheim, 1968.
 \bibitem{CS}  B. Carl and I. Stephani, \textit{Entropy, compactness and the approximation of operators}, Cambridge University Press, 1990.
\bibitem{CP} F. Cobos and L. Persson, \textit{Real interpolation of compact operators between quasi-Banach spaces}, Mathematica Scandinavica, Vol. 82, No. 1 (1998), pp. 138-160
\bibitem{Delo} R. A. DeVore and G. G. Lorentz, \textit{Constructive approximation},  No. 303, Spring-Verlag, New York, 1993.
\bibitem{GKL} I. Gohberg, M.A. Kaashoek, and D.C. Lay, \textit{Equivalence, Linearization, and decomposition of holomorphic operator functions}, J. Funct. Anal. \textbf{28}, 102-144 (1978)
\bibitem{KV} M. A. Kaashoek and S. M. Verduyn Lunel, \textit{Completeness Theorems and Characteristic Matrix Functions. Applications to Integral and Differential Operators.Operator Theory: Advances and Applications 288,} Birkhauser.
\bibitem{Mar} A. S. Markus, \textit{Some criteria for the completeness of a system of root vectors of a linear operator in a Banach space}, Mat. Sb. \textbf{70} (112) (1966), 526-561; English transl., Amer. Math. Soc. Transl. (2) \textbf{85} (1969), 51-91. MR. 35 7151
\bibitem{Pi} P. Pietsch, \textit{Approximation spaces}, J. Approx. Theory, 32 (1981), no. 2, 115-134. 
 \bibitem{Pieeig} A. Pietsch, \textit{Eigenvalues and s-numbers}, Cambridge studies in advanced mathematics 13, 1985.

 


\end{thebibliography}

\end{document}